\documentclass[12pt]{article}

\usepackage{amssymb,amsmath,color,graphicx,dsfont,mathrsfs,amsthm,stmaryrd}
\usepackage{mathtools}
\usepackage{authblk}
\usepackage{verbatim}
\usepackage{caption}
\usepackage{hyperref}
\usepackage{siunitx}

\newtheorem{theorem}{Theorem}

\newtheorem{lemma}{Lemma}

\newtheorem{remark}{Remark}

\DeclareMathOperator{\dive}{div}

\newcommand{\R}{\mathbb{R}}

\newcommand{\al}{\alpha}
\newcommand{\eps}{\varepsilon}

\renewcommand{\j}{{\bf j}}
\newcommand{\jone}{{\bf j}_1}
\newcommand{\jtwo}{{\bf j}_2}
\renewcommand{\L}{{\bf L}}
\newcommand{\Leps}{{\bf L_\eps}}
\newcommand{\x}{{\bf x}}

\newcommand{\y}{{\bf y}}
\newcommand{\n}{{\bf n}}
\newcommand{\B}{{\bf B}}
\newcommand{\BS}{{\bf BS}}
\newcommand{\e}{{\bf e}}
\newcommand{\h}{{\bf h}}
\newcommand{\z}{{\bf z}}
 
\author[1]{R\'emi Robin\thanks{remi.robin@inria.fr}$^,$ }
\author[2]{Francesco Volpe \thanks{francesco.volpe@renfusion.eu}}

\affil[1]{\footnotesize Laboratoire Jacques-Louis Lions, Sorbonne Universit\'e}
\affil[2]{\footnotesize Renaissance Fusion, https://stellarator.energy}

\title{Minimization of magnetic forces on Stellarator coils}
\begin{document}
\maketitle
\begin{abstract}
    Magnetic confinement devices for nuclear
    fusion can be large and expensive. Compact stellarators are promising  
    candidates for cost-reduction, but introduce new difficulties: 
    confinement in smaller volumes requires higher magnetic field, which calls
    for higher coil-currents and ultimately causes higher Laplace forces on the
    coils - if everything else remains the same. 
    This motivates the inclusion of force reduction in
    stellarator coil optimization. 
    In the present paper
    we consider a coil winding surface, we prove that
    there is a natural and rigorous way to define the Laplace force (despite
    the magnetic field discontinuity across the current-sheet), 
    we provide examples of cost associated (peak force,
    surface-integral of the force squared) and discuss easy generalizations
    to parallel and normal force-components, as these will be subject to
    different engineering constraints. 
    Such costs can then be easily added to the 
    figure of merit in any multi-objective stellarator coil optimization
    code. We demonstrate this for a generalization of the 
    {\tt REGCOIL} code \cite{Landreman}, which we rewrote in python,
    and provide numerical examples for the NCSX (now QUASAR) design.
    We present results for various definitions of the cost function, 
    including peak force reductions by up to 40 \%,
    and outline future work for further reduction.
\end{abstract}

\section{Introduction}
Stellarators are non-axisymmetric
toroidal devices that magnetically confine fusion plasmas 
\cite{Helander}. 
Thanks to specially shaped coils they do not require a current in the plasma,
hence are more stable and steady-state than tokamaks. However, they exhibit
comparable confinement, hence tend to be about as large. Like tokamaks,
confinement can be improved (and size reduced) by adopting stronger magnetic
fields.

Fields as high as 8-12 T were only tested in two series of
tokamak experiments at the MIT and ENEA, culminated respectively
in Alcator C-mod \cite{Marmar} and FTU \cite{Pucella}.
For comparison, ITER has a field of 5.3 T on axis.
Other high-field tokamaks were designed but not built \cite{Coppi,Meade},
although the new high-field SPARC tokamak has been designed, modeled and its
construction is expected to start in 2021 \cite{Creely}. 

For stellarators and heliotrons, there is broad agreement that power-plants
will require at least 4-6 T \cite{Sagara}, but  
fields as high as 8-12 T have only been proposed very recently
\cite{Queral}. Two private companies are working toward that goal
\cite{TypeOne,RenFus}. 

Generating strong fields requires high currents and of course results in high
forces on the coils (unless their design is modified, as we will argue 
in this paper). Up to 5 T, the issue can be resolved by adequately
reinforced coil-support structures and coil-spacers \cite{Schauer}. 
However, a further increase to 10 T will result in 4$\times$ higher forces.
This calls for including force-reduction in the coil design and
optimization process, along with other criteria.

Such need was recognized earlier on for heliotrons, and spurred reduced force  
(so-called force-free) heliotron designs \cite{Imagawa}. 
From a mathematical standpoint this is not too surprising, since helical
fields in heliotrons resemble the eigenfunctions of the curl operator on a
torus \cite{AlonsoRodriguez}: $\nabla \times {\bf B} = \lambda {\bf B}$. 
This, combined with Maxwell-Ampere law, implies that $\bf B$ and the current
  density ${\bf j}$ are parallel, and there are no Laplace forces on the coils. 

  Modular coils for advanced stellarators, on the other hand, are the result
  of numerical optimization. 
The most common optimization criterion is to reproduce the
target magnetic field to within one part in $10^4$ or $10^5$. 
Typically this is solved on a 2-D toroidal surface conformal to the
plasma boundary, called Coil Winding Surface (CWS).  On that surface,
numerical codes compute the current potential (thus, ultimately, the
current pattern) that best reproduces the target plasma boundary, in a
least-squares sense \cite{Landreman}. This is the principle of the
seminal {\tt NESCOIL} code \cite{Merkel}. Further developments included
engineering-constrained nonlinear optimizers \cite{Strickler} and the
Tikhonov-regularized {\tt REGCOIL} \cite{Landreman}.
The latter includes the squared coil-current density in the objective function,
which leads to more ``gentle'', easier-to-build coil shapes. 
All these codes fix the CWS; 
more recently, a free-CWS 3-D search method was developed \cite{Zhu}. 

In the present article, we generalize {\tt REGCOIL} to include coil-force reduction.
This is obtained by adding a third term to the objective function, quantifying
the Laplace forces on the CWS. Several metrics are possible, for example
the surface-integral of the squared Laplace force, or the peak value of the
Laplace force.
We recall that the Laplace forces are a self-interaction $\L$ of a 
surface-current (of density $\bf j$)
with itself. To that end, first we introduce the force $\L$ exerted by a
surface-current of density $\jone$ on a
surface-current of density $\jtwo$.

The paper is organized as follows. 
The Laplace force on a current-sheet is rigorously derived in
Eqs.~\ref{eq:L1}-\ref{eq:L4} of Section \ref{sec:def_Laplace_force}. 
The expression obtained is not overly expensive from a numerical standpoint.
Based on that, possible cost functions are proposed and briefly discussed in 
Section \ref{sec:costs}. Finally 
section \ref{sec:numerics} illustrates the numerical results
obtained with the two main cost-functions for the quasi-axisymmetric
stellarator design formerly known as NCSX, then QUASAR,
including a reduction of the peak force by up to $40 \%$.

\section{Laplace force on a surface}
\label{sec:def_Laplace_force}
\subsection{Limit definition of Laplace force exerted by a
  current-sheet on itself}
In the following, $S$ denotes the CWS and $\n$ the unit vector field normal to
$S$ and pointing outward. 
$\j$ is a vector field on $S$, representing the surface-current
density, i.e.~the current per unit length (not per unit surface, as is usually
the case for this notation). 
The Laplace force is the magnetic component of the Lorentz
force; the Laplace force per unit {\em surface} (not per unit volume) is given
by ${\bf F}=\j \times {\bf B}$, although here, quite often, it will simply
be called 'force', for brevity.

It is well-known that a surface-current causes a discontinuity in the
tangential component of the magnetic field, given by the interface
condition $\n_{12}\times ({\bf H_2}-{\bf H_1})=\j$. 
The resulting jump in the tangential component of $\B$ results in a normal
force wherever $\j \not =0$. That force, 
proportional to $ |\j| ^2 $, tries to increase the thickness of the
CWS. To ensure that that force remains reasonably small, 
one can easily add a cost $|\j|_{L^2}$ or $|\j|_\infty$ to the
multi-objective figure-of-merit or optimize under a constraint on
$\j$.

>From now on, though, we will focus on the other contributions to the
Laplace force. We define them in a location $\y\in S$ as follows:
\[
{\bf F}(\y)=
\lim_{\eps \to 0} \frac{1}{2} \{
\j(\y)\times \big[ \B(\y+\eps \n(\y))+\B(\y-\eps \n(\y)) \big]
\}.
\]
Let us focus on the case where $\B$ is only generated by currents on $S$; there
are no permanent magnets nor magnetically susceptible media. 
In any $\y \not \in S$, the field is given by the Biot-Savart law in vacuo:  
\begin{equation}
  \B(\y)={\bf \BS}(\j)(\y)=\int_S \j(\x)\times \frac{\y-\x}{|\y-\x|^3}dS(\x),
  \label{eq:BS}
\end{equation}
where, to reduce the amount of notation, we dropped the
$\frac{\mu_0}{4 \pi}$ factor in front of the integral.
The notation ${\bf \BS}(\j)$ refers to the Biot-Savart operator,
function of $\j$, that maps each $\y \not \in S$ in the local field, $\B(\y)$.

\begin{remark}
  ${\bf \BS}(\j)$ cannot be defined on $S$, unless $\j=0$.
  This is a consequence of $\frac{1}{|\y-\x|^2}$ not being
  integrable for $\y\in S$.
  It was expected because the interface condition introduces a discontinuity
  of ${\bf \BS} (\j)(\y+\eps \n(\y))$ at $\eps=0$.
    \label{rmk:B_not_integrable}
\end{remark}
However, since ${\bf \BS}(\j)$ is well-defined in locations $\not \in S$,
we can define for any $\y \in S$ and any $\eps>0$ the bilinear map
\begin{equation}
    \Leps(\jone,\jtwo)(y)= \frac{1}{2}
    \{
    \jone(\y)\times \big[ \BS(\jtwo)(\y+\eps \n(\y))+\BS(\jtwo)(\y-\eps \n(\y)) \big]
    \}.
    \label{eq:L_eps}
\end{equation}
This describes the Laplace force that a surface-current of
density $\jtwo$ exerts on another of density $\jone$, per unit surface. 
Since we are dealing with a stellarator, these currents are constant in time and
there is no need to include induced fields and the associated forces.

The `average Laplace force' that a current of density $\j$ exerts 
on point $\y\in S$ is thus $\L(\j)(\y)=\lim_{\eps \to 0} \Leps(\j,\j)(\y)$.

This definition, however, raises several questions:
\begin{enumerate}
    \item \label{Q1}Under which assumptions on $\j$ can we ensure that $\L(\j)$ is well defined (i.e., that the limit is well -defined)?
    \item \label{Q2}Can we find an explicit expression of $\L(\j)$ (i.e., without a limit on $\eps$)?
    \item \label{Q3}Which functional space does $\L(\j)$ belong to (for $\j$ in a given functional space)?
\end{enumerate}

The first point is more theoretical, but is necessary
to answer the second and third one, which have very practical
consequences. Indeed, without an explicit expression for $\L(\j)$, the
numerical computation of the Laplace force may be a complex matter. A
typical approach would involve 3 different scales. From the smallest to the
largest, these are the discretisation-length of $S$, $h$, 
the infinitesimal displacement $\eps$, and the characteristic
distance of variation of the magnetic field, $d_B$.

An accurate computation of $\B(\y+\eps \n(\y))$ requires $S$ to be finely
discretized, with $h \ll \eps$.
This is because $\int_S |\y+\eps \n(\y)-\x|^{-2}dS(\x)$ blows up when
$\eps \to 0$.
Indeed when we replace the integral with a discrete sum and take the limit for 
small $\eps$,
$$\tilde \B\big( \y_{i,j}+\eps \n(\y_{i,j}) \big)=\sum_{k,l}\j(\x_{k,l})\times \frac{\y_{i,j}+\eps \n(\y_{i,j})-\x_{k,l}}{|\y_{i,j}+\eps \n(\y_{i,j})-\x_{k,l}|^3} \overset{\eps \to 0}{\approx} \frac{\j(\y_{i,j})\times \n(\y_{i,j})}{\eps^2}$$
which for small $\eps$ diverges like $\eps^{-2}$,
as shown in Figure \ref{fig:B_norm} for NCSX.
\begin{figure}
    \includegraphics[width=0.6\textwidth]{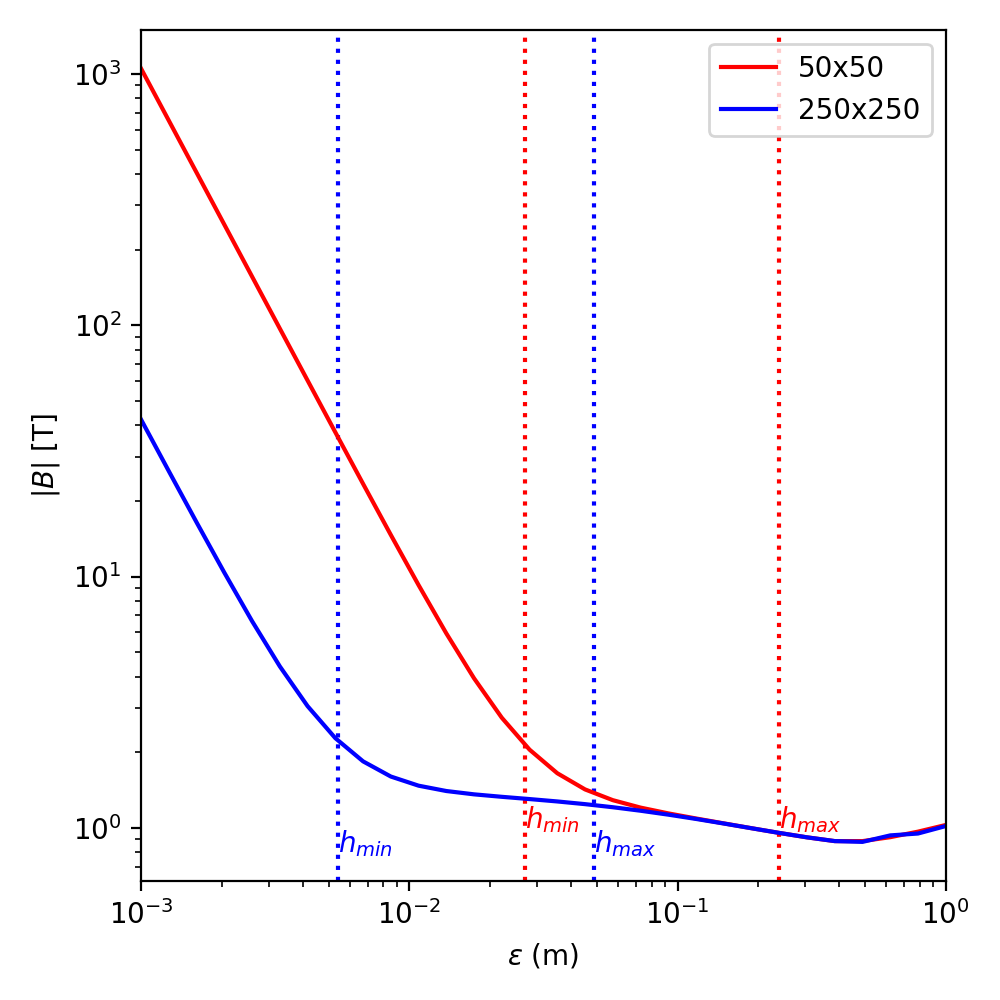}
    \captionof{figure}{Average norm of $\B$ as a function of the distance $\eps$ from the surface $S$, for two different girds, more coarse (red) or fine (blue). The divergence for small $\eps$ is a numerical artifact due to the numerical discretisation when $\eps \lesssim h$.}
    \label{fig:B_norm}
\end{figure}

The semi-sum $\frac{\B(\y+\eps \n(\y))+\B(\y-\eps \n(\y))}{2}$ is numerically
more stable, but we still need $h \lesssim \eps$ (as it will be shown later in
Figure \ref{fig:L_eps}). Such fine mesh makes it costly to accurately compute 
$\L(\j)(\y)$ as $\lim_{\eps \to 0} \Leps(\j,\j)(\y)$.

The functional space of $\L(\j)$ is also important to understand what type of
penalization can be applied to minimize this force, or a related metric.

\subsection{Notations}  \label{subsec:notations}
We start by introducing the following notations:
\begin{itemize}
    \item $S$ is a smooth 2-dimensional Riemannian submanifold of $\R^3$, diffeomorphic to the 2-torus.
    \item $\mathfrak{X}(S)$ is the set of smooth vector fields on $S$.
    \item $\langle X \cdot Y \rangle$ denote the scalar product (in $\R^3$) between the vector fields $X$ and $Y$.
    When both vector are tangent to $S$, we sometime denote $\langle X \cdot Y \rangle_{T_xS}$ the scalar product at $x\in S$ (which coincides with the one in $\R^3$).
    \item $L^p(S)$ and $H^1(S)$ are the Hilbert spaces defined as the completion of $\mathcal{C}^\infty(S)$ for the norms
    \begin{align*}
        |f|_{L^p(S)}=&\big( \int_S f^p dS \big)^{1/p}\\
        |f|^2_{H^1(S)}=&\int_S \big( f^2 +\langle \nabla f \cdot \nabla f \rangle \big)dS.
    \end{align*}
    \item $\mathfrak{X}^{p}(S)$ and $\mathfrak{X}^{1,2}(S)$ are the Hilbert spaces defined as the completion of $\mathfrak{X}(S)$ for the norms
    \begin{align*}
        |\j|_{\mathfrak{X}^p(S)}=&\big| \sqrt{\j_x^2+\j_y^2+\j_z^2} \big|_{L^p(S)}\\
        |\j|_{\mathfrak{X}^{1,2}(S)}=&\sqrt{|\j_x|_{H^1(S)}^2+|\j_y|_{H^1(S)}^2+|\j_z|_{H^1(S)}^2}
    \end{align*}
    where $\j_x,\j_y$ and $\j_z$ are the components of $\j$ in $\R^3$ for an arbitrary orthogonal basis.
  \item The spaces $L^p(S,\R^3)$ and $H^{1,2}(S,\R^3)$ are related to $\mathcal{C}^\infty(S,\R^3)$ in the same way as $\mathfrak{X}^p(S)$ and $\mathfrak{X}^{1,2}(S)$ are related to $\mathfrak{X}(S)$. 
    \item $\pi$ is the projector on the tangent bundle. For any ${\bf Y} \in \mathcal{C}^\infty(S,\R^3)$, we define
    \begin{equation}
        \label{eq:pi_x}
        \forall \x \in S,(\pi({\bf Y}))_\x={\bf Y}_\x-\langle {\bf Y}_\x \cdot \n(\x)\rangle \n(\x)\in T_xS.
    \end{equation}
    Since $\pi({\bf Y})$ is clearly a tangent vector field on $S$, it belongs to $\mathfrak{X}(S)$.
\end{itemize}

\subsection{Theorem: computing the Laplace force exerted by one 
  current-sheet on another}  \label{subsec:Theo}
\begin{theorem}
    \label{th:main}
    Let $\jone,\jtwo \in \mathfrak{X}^{1,2}(S)$. Then $\Leps(\jone,\jtwo)$ has
    an $\eps \to 0$ limit in
    $L^p(S,\R^3)$, for any $1\leq p <\infty$, denoted 
    $\L(\jone,\jtwo)$. Furthermore, $\L$ is a continuous bilinear map
    $\mathfrak{X}^{1,2}(S) \times \mathfrak{X}^{1,2}(S) \to L^p(S,\R^3)$ given by
    \begin{align}
        \label{eq:L1}
        \L(\jone,\jtwo)(\y)=&-\int_S \frac{1}{|\y-\x|} \big[ \dive_\x(\pi_\x \jone(\y)) +\pi_\x \jone(\y) \cdot \nabla_\x \big] \jtwo(\x) d\x \\
    \label{eq:L2}&+\int_{S} \langle \jone(\y)  \cdot \n(\x)\rangle \frac{\langle \y-\x \cdot \n(\x)\rangle}{|\y-\x|^3} \jtwo(\x) d\x \\
    \label{eq:L3}&+ \int_S \frac{1}{|\y-\x|} \big[ \langle \jone(\y) \cdot \jtwo(\x) \rangle \dive_\x(\pi_\x)+  \nabla_\x \langle \jone(\y) \cdot \jtwo(\x) \rangle \big]  d\x  \\
    \label{eq:L4}&- \int_S \langle \jone(\y) \cdot \jtwo(\x) \rangle \frac{\langle \y-\x \cdot \n(\x) \rangle}{|\y-\x|^3}\n(\x) d\x
    \end{align}
\end{theorem}
\begin{remark}
    \label{rmk:abuse_notation}
    \begin{itemize}
    \item The notation $V \cdot \nabla_\x F$
      (where $V\in \mathfrak{X}(S)$ is a 2D vector
      and $F\in \mathcal{C}^\infty(S,\R^3)$ a 3D one) stands for
      $\sum_{\al=1}^2 \sum_{i=1}^3 V^\al \partial_\al F^i \e_i$. 
          Here $\al$ is the index for the surface coordinates
          ($\theta$ and $\varphi$, for example), whereas $(\e_1,\e_2,\e_3)$ is
          a basis of $\R^3$.
            \item $\dive_\x(\pi_\x)$ stands for the 3D vector $\sum_{i=1}^3 \dive_\x(\pi_\x \e_i)\e_i$
    \end{itemize}
\end{remark}

The proof of the theorem is somewhat long and will be organized as follows: 
in Sec.~\ref{subsec:gen} we will rewrite Eq.~\ref{eq:L_eps} in terms
of surface integrals \ref{eq:surfint1}, \ref{eq:surfint2},
\ref{eq:surfint3} and \ref{eq:surfint4}. 
Integrals \ref{eq:surfint1} and \ref{eq:surfint3}, with integrands tangential
to $S$ (``tangential terms'') will be dealt with in Sec.~\ref{subsec:tang1} and
\ref{subsec:tang2}. 
Integrals \ref{eq:surfint2} and \ref{eq:surfint4}, with integrands normal
to $S$ (``normal terms'') will be treated in Sec.~\ref{subsec:norm1} and
\ref{subsec:norm2}.

\subsubsection{Proof: beginning and general ideas} \label{subsec:gen}
Consider two linear densities of surface-currents
$\jone,\jtwo \in \mathfrak{X}^{1,2}(S)$ and fix $\epsilon>0$.

Thanks to the well-known formula
$\bf A\times (B\times C)= (A\cdot C)B- (A \cdot B) C$, we obtain from
Eqs.\ref{eq:BS} and \ref{eq:L_eps} that
\begin{align}
    \label{eq:Leps1}
    \Leps(\jone,\jtwo)(y)=&\int_S \langle \jone(\y) \cdot (\frac{\y-\x+\eps \n(\y)}{2|\y-\x+\eps \n(\y)|^3}+\frac{\y-\x-\eps \n(\y)}{2|\y-\x-\eps \n(\y)|^3})\rangle \jtwo(\x) d\x \\
    \label{eq:Leps2}
    &-\int_S \langle \jone(\y) \cdot \jtwo(\x) \rangle (\frac{\y-\x+\eps \n(\y)}{2|\y-\x+\eps \n(\y)|^3}+\frac{\y-\x-\eps \n(\y)}{2|\y-\x-\eps \n(\y)|^3}) d\x.
\end{align}
The difficulty is that $\frac{1}{|x|^2}$ is not integrable in 2 dimensions (Remark \ref{rmk:B_not_integrable}).
Hence, it does not make sense to take the limit for $\eps \to 0$ 
directly inside the integral.
Nevertheless, we can use the following equality:
\begin{align}
    \label{eq:T1}
     \int_{S} \langle \jone(\y) \cdot \frac{\y-\x+\eps \n(\y)}{|\y-\x+\eps \n(\y)|^3}\rangle \jtwo(\x) d\x
    =\int_{S} \langle \jone(\y) \cdot \nabla_\x \frac{1}{|\y-\x+\eps \n(\y)|}\rangle \jtwo(\x) d\x
\end{align}
where $\nabla_\x$ is the gradient in $\R^3$ with respect to the variable $\x$.
We would like to integrate by part to take advantage of the integrability of $\frac{1}{|x|}$.
For this we decompose $\nabla_\x$ into the tangential part of the gradient, $\nabla_S$, and normal component, $\nabla_{\perp}$.
As a consequence, for Eq.~\ref{eq:Leps1} we have the following equalities:
\begin{align}
    &\int_{S} \langle \jone(\y) \cdot \frac{\y-\x\pm \eps \n(\y)}{|\y-\x\pm \eps \n(\y)|^3}\rangle \jtwo(\x) d\x\\
    =&\int_{S} \langle \jone(\y) \cdot \nabla_\x \frac{1}{|\y-\x\pm \eps \n(\y)|}\rangle \jtwo(\x) d\x\\
    =& \int_{S} \langle \jone(\y) \cdot \nabla_S \frac{1}{|\y-\x\pm \eps \n(\y)|}\rangle \jtwo(\x) d\x
\label{eq:surfint1}
    \\
    &+ \int_S \langle \jone(\y) \cdot \frac{\langle \y-\x, \n(\x) \rangle \pm \eps \langle \n(\y), \n(\x) \rangle}{|\y-\x\pm \eps \n(\y)|^3}\n(\x)\rangle \jtwo(\x) d\x
    \label{eq:surfint2}
\end{align}
and for expression \ref{eq:Leps2}
\begin{align}
    &\int_S \langle \jone(\y) \cdot \jtwo(\x) \rangle \frac{\y-\x\pm \eps \n(\y)}{|\y-\x\pm \eps \n(\y)|^3} d\x \\
    =&\int_S \langle \jone(\y) \cdot \jtwo(\x) \rangle \nabla_\x \frac{1}{|\y-\x\pm \eps \n(\y)|} d\x\\
    =& \int_S \langle \jone(\y) \cdot \jtwo(\x) \rangle \nabla_S \frac{1}{|\y-\x\pm \eps \n(\y)|} d\x
        \label{eq:surfint3}
    \\
    &+ \int_S \langle \jone(\y) \cdot \jtwo(\x) \rangle \frac{\langle \y-\x, \n(\x) \rangle \pm \eps \langle \n(\y), \n(\x) \rangle}{|\y-\x\pm \eps \n(\y)|^3} \n(\x) d\x
    \label{eq:surfint4}
\end{align}

\subsubsection{Proof: First tangential term} \label{subsec:tang1}

Integration by parts on a compact manifold $\mathcal{M}$ without boundary is given by the following formula. 
Let $f\in \mathcal{C}^\infty(\mathcal{M})$ and ${\bf X}$ a smooth vector field on $\mathcal{M}$, then
\begin{equation}
   \int_{\mathcal{M}} \dive(f{\bf X})= 0 =
   {\bf X}f + \int_{\mathcal{M}} f \dive {\bf X} 
   \label{eq:intparts}
\end{equation}
We also recall that ${\bf X}f=\langle {\bf X} \cdot \nabla f\rangle$ in Euclidean coordinates.

Let us start with the first tangential term (Eq.~\ref{eq:surfint1}):
\begin{align*}
    &\int_{S} \langle \jone(\y) \cdot \nabla_{S} \frac{1}{|\y-\x \pm\eps \n(\y)|}\rangle_{\R^3} \jtwo(\x) d\x \\
  =& \int_{S} \langle \pi_\x \jone(\y)  \cdot \nabla_{S} \frac{1}{|\y-\x \pm\eps \n(\y)|}\rangle_{T_xS} \jtwo(\x) d\x \
\end{align*}
as $\jone(\y)-\pi_\x \jone(\y) \propto \n(\x)$. 

Then, let $j_2^i(\x)$ be the
$i$-th component in $\R^3$ of $\jtwo$.
Using integration by parts (Eq.~\ref{eq:intparts}), 
the $i$-th component of the last integral writes 
\begin{align}
  &\int_{S} \langle j^i_2(\x) \pi_\x \jone(\y)  \cdot \nabla_{S} \frac{1}{|\y-\x \pm\eps \n(\y)|}\rangle_{T_xS} d\x
  \label{eq:dem1}
    \\
    = &
    -\int_{S} \frac{1}{|\y-\x \pm\eps \n(\y)|} \dive_\x (j^i_2(\x) \pi_\x \jone(\y))d\x
    \label{eq:dem2}
    \\
=&-\int_S \frac{1}{|\y-\x \pm\eps \n(\y)|} \big[ j^i_2(\x) \dive_\x (\pi_\x \jone(\y)) + \langle \pi_\x \jone(\y) \cdot \nabla j^i_2(\x) \rangle \big] d\x
    \label{eq:dem3}
\end{align}
Thus the term in equation $\ref{eq:T1}$ is equal to:
$$- \sum_{i=1}^3 \big(\int_S \frac{1}{|\y-\x \pm\eps \n(\y)|} \big[ j^i_2(\x) \dive_\x (\pi_\x \jone(\y)) + \langle \pi_\x \jone(\y) \cdot \nabla j^i_2(\x) \rangle \big] d\x \big) \e_i$$
that, with the conventions introduced in Remark \ref{rmk:abuse_notation},
can be rewritten as:
$$-\int_S \frac{1}{|\y-\x \pm\eps \n(\y)|} \big[ \dive_\x (\pi_\x \jone(\y)) +\pi_\x \jone(\y) \cdot \nabla_\x \big]  \jtwo(\x) d\x$$
Now, we will prove that it is possible to take the limit $\eps \to 0$ inside this integral.
The first step is to use the following estimate.
\begin{lemma}
    $|\dive_\x  \pi_\x \jone(\y)| \leq C(S) |\jone(\y)|$ with $C(S)$ a constant that only depends on the metric of $S$.
\end{lemma}
\begin{proof}
    Indeed, the application \begin{align*}
        \dive_\x \pi_\x \colon \R^3 & \to  \mathcal{C}^{\infty}(S)\\
        \bf{v} &\mapsto \big( \x \mapsto \dive_\x (\pi_\x \bf{v}) \big).
      \end{align*} is a continuous linear application.
\end{proof}
We also need a Young-type inequality for 2-dimensional compact manifolds.
\begin{lemma}
    \label{HLS}
    For all $1\leq q< \infty$, there exists $C>0$ such that for all $f$ in $L^2(S)$, $|\int_S \frac{1}{|\y-\x|} f(\x) d\x|_{L^q_y} \leq C_q |f|_{L^2}$.
\end{lemma}
\begin{proof}
    Let $d_g$ denote the Riemannian distance on $S$. By a Hardy-Littlewood-Sobolev inequality $\int_S \frac{1}{d_g(\y,\x)}f(x) d\x$ is in $L^q(S)$ for all $1\leq q < \infty$.
    This result can be found, for example, in \cite{han_hardylittlewoodsobolev_2016} or can be proved directly with the arguments of the proof of the classical Young inequality.
    As the Euclidean distance and the Riemannian distance are equivalent, the lemma is proved.
\end{proof}
Thus, for all $1\leq q< \infty$, $\int_S \frac{1}{|\y-\x|}  \partial_i \jtwo(\x) d\x \in L^q(S,\R^3)$.
Besides, by Sobolev embedding \cite{hebey_nonlinear_2000}, there is a continuous injection $\mathfrak{X}^{1,2}(S)\xhookrightarrow{} \mathfrak{X}^p(S)$ for all $1\leq p<\infty$.
As a consequence $j_1^i(\y) \int_S \frac{1}{|\y-\x|}\partial_i \jtwo(\x) d\x \in \mathfrak{X}^p(S)$ for $1\leq p<\infty$.

With these estimates we can conclude, using dominated convergence, that:
\begin{align}
  &\int_{S} \langle \jone(\y) \cdot \nabla_{S} \frac{1}{|\y-\x \pm\eps \n(\y)|}\rangle \jtwo(\x) d\x
  \label{eq:dem4}
  \\
  \xrightarrow{\mathfrak{X}^p(S)} & -\int_S \frac{1}{|\y-\x|} \dive_\x (\pi_\x \jone(\y)) \jtwo(\x) d\x
  \label{eq:dem5}
  \\
  &-\int_S \frac{1}{|\y-\x|} (\pi_\x \jone(\y) \cdot \nabla_\x) \jtwo(\x) d\x
  \label{eq:dem6}
\end{align}
Note that the two integrals (respectively
with the sign $+$ or $-$ at denominator) converge to the same limit. 
Their sum yields the integral on the
right-hand-side of Eq.~\ref{eq:L1}.

\subsubsection{Proof: Second tangential term} \label{subsec:tang2}

Now, let us tackle the  term
in Eq.~\ref{eq:surfint3}. 
We start by computing the $i$-th component of that integral, i.e.~its projection
on $\e_i$, then follow a derivation similar to Eqs.~\ref{eq:dem1}-\ref{eq:dem3}: 
\begin{align*}
    &\int_S \langle \jone(\y) \cdot \jtwo(\x) \rangle  \langle \e_i \cdot \nabla_S \frac{1}{|\y-\x \pm\eps \n(\y)|} \rangle d\x \\
    =&\int_S \langle \jone(\y) \cdot \jtwo(\x) \langle  \pi_\x \e_i \cdot \nabla_S \frac{1}{|\y-\x \pm\eps \n(\y)|}\rangle d\x\\
    =&-\int_S \frac{1}{|\y-\x \pm\eps \n(\y)|} \dive_\x (\langle \jone(\y) \cdot \jtwo(\x) \rangle \pi_\x \e_i)  d\x\\
    =&-\int_S \frac{1}{|\y-\x \pm\eps \n(\y)|} \big[ \langle \jone(\y) \cdot \jtwo(\x) \rangle \dive_\x (\pi_\x \e_i)+ \langle \pi_\x \e_i \cdot \nabla_\x \langle \jone(\y) \cdot \jtwo(\x) \rangle \rangle \big]  d\x.
\end{align*}
Using the notation of Remark \ref{rmk:abuse_notation}, we
find the vector form of the integral:  
$$-\int_S \frac{1}{|\y-\x \pm\eps \n(\y)|} \left(\langle \jone(\y) \cdot \jtwo(\x) \rangle \dive_\x (\pi_\x)+  \nabla_\x \langle \jone(\y) \cdot \jtwo(\x) \rangle \right)  d\x.$$
Due to the same arguments invoked in Eqs.~\ref{eq:dem4}-\ref{eq:dem6}, 
both integrals, with the sign $+$ and $-$ at denominator, 
converge in $\mathfrak{X}^p(S)$ to the same limit, 
$$-\int_S \frac{1}{|\y-\x|} \big[ \langle \jone(\y) \cdot \jtwo(\x) \rangle \dive_\x (\pi_\x)+  \nabla_\x \langle \jone(\y) \cdot \jtwo(\x) \rangle \big]  d\x.$$
Their sum yields integral \ref{eq:L3}.

\subsubsection{Proof: First normal term} \label{subsec:norm1}

Let us now focus on the normal component of \ref{eq:Leps1},
namely Eq.~\ref{eq:surfint2}. This is in effect the sum of two integrals, which
we will discuss separately:
\begin{align}
  &\int_{S} \langle \jone(\y)  \cdot \n(\x)\rangle
  \frac{\langle \y-\x,\n(\x)\rangle}{|\y-\x \pm\eps \n(\y)|^3} \jtwo(\x) d\x
  \\
  &\int_{S} \langle \jone(\y)  \cdot \n(\x)\rangle
  \frac{\pm\eps \langle \n(\y), \n(\x) \rangle }{|\y-\x \pm\eps \n(\y)|^3} \jtwo(\x) d\x 
\end{align}

First notice that we have the following estimate:
\begin{lemma}
    \label{lem:reg1} $\exists C >0, \forall \x \not =\y\in S, \frac{|\langle \y-\x,\n(\x)\rangle|}{|\y-\x|^2}\leq C$.
\end{lemma}
\begin{proof}
    Let us suppose there exist two sequences $(\x_n),(\y_n)$ in $S$ such that $\x_n \not = \y_n$ and $\frac{|\langle \y_n-\x_n,\n(\x_n)\rangle|}{|\y_n-\x_n|^2}\to \infty$.
    Up to an extraction, we can suppose that $\x_n \to \x_0\in S$. If $\y_n$ does not converges to $\x_0$, we can extract a subsequence such that $\frac{|\langle \y_n-\x_n,\n(\x_n)\rangle|}{|\y_n-\x_n|^2}$ does not diverge. This is a contradiction, hence both $\x_n$ and $\y_n$ converge to $\x_0 \in S$.
    
    Let $\Gamma(\x,\y)=\langle \y-\x,\n(\x)\rangle$. As $S$ is smooth, so is $\Gamma$. Its partial differentials are 
    \begin{align*}
        \forall \h \in T_xS, \quad d_\x\Gamma_{\x,\y}(\h)&=-\langle -\h \cdot \n(\x)\rangle + \langle \y-\x \cdot d\n_x(\h) \rangle \\
        \forall \h \in T_yS, \quad d_\y\Gamma_{\x,\y}(\h)&=\langle \h \cdot \n(\x)\rangle
    \end{align*}
    Thus, at the point $(\x_0,\x_0)$, both first derivatives vanish. As a consequence, for $n$ big enough there exists $C>0$ such that $\Gamma(\x_n,\y_n)\leq C |\x_n-\y_n|^2$, contradiction.
\end{proof}
Now, we need to find a minoration of $|\y-\x\pm \eps \n(\y)|$.
\begin{lemma}
    \label{lem:min}
    For $\eps$ small enough, for all $\mu>0$, $$|\x-\y\pm \eps \n(\y)|^\mu \geq \max \big( (\frac{1}{\sqrt{2}}|\x-\y|)^\mu,\eps^\mu \big).$$
\end{lemma}
\begin{proof}
    \begin{align*}
        |\y-\x\pm \eps \n(\y)|^2=& |\y-\x|^2 \pm \eps \langle \y-\x,\n(\y)\rangle + \eps^2&\\
        \geq& |\y-\x|^2 - C\eps |\y-\x|^2 + \eps^2 \quad & \text{by lemma \ref{lem:reg1}.}
    \end{align*}
    Thus for $\eps\leq 1/(2C)$, we have, $\forall \mu>0$,
    $$|\x-\y\pm \eps \n(\y)|^\mu \geq \max \big( (\frac{1}{\sqrt{2}}|\x-\y|)^\mu,\eps^\mu \big).$$    
\end{proof}
Using Lemmas \ref{lem:reg1} and \ref{lem:min}, for some constant C, $\big |\frac{\langle \y-\x,\n(\x)\rangle}{|\y-\x \pm\eps \n(\y)|^3} \big|$ is dominated by $C\frac{1}{|\y-\x|}$, which is integrable. By dominated convergence
\begin{equation}
  \int_{S} \langle \jone(\y)  \cdot \n(\x)\rangle \frac{\langle \y-\x,\n(\x)\rangle}{|\y-\x \pm\eps \n(\y)|^3} \jtwo(\x) d\x \xrightarrow{\mathfrak{X}^p(S)}\int_{S} \langle \jone(\y)  \cdot \n(\x)\rangle \frac{\langle \y-\x,\n(\x)\rangle}{|\y-\x|^3} \jtwo(\x) d\x,
\label{eq:Pluto}
\end{equation}
i.e.~we obtained integral \ref{eq:L2}.

Now we have to deal with $\frac{\eps\langle \n(\y), \n(\x) \rangle}{|\y-\x \pm\eps \n(\x)|^3}$, but we will show their net contribution to converge to zero.

To begin with, we could use the smallness of the term
$\langle \jone(\y) \cdot \n(\x)\rangle$ to ensure integrability.
Instead, we will prove the following lemma which will also be useful later.
Let $\Delta=\{(\z,\z) \mid \z\in S\}\subset S^2$.
\begin{lemma}
    \label{lem:reg2}
    Let $f_\eps: S^2 \setminus \Delta \ni (\x,\y) \mapsto  \frac{1}{|\y-\x +\eps \n(\y)|^3}-\frac{1}{|\y-\x-\eps \n(\y)|^3} d\x$.
    Then $\exists \eta>0,\exists M>0$, $\forall \alpha \in (-0.5,3.5), \forall \eps <\eta$, $\forall (\x,\y)$, $|\eps^{\alpha} f_\eps(\x,\y)|\leq M\frac{1}{|\x-\y|^{5/2-\alpha}}$. 
\end{lemma}
\begin{proof}
    
    \begin{align*}
        f_{\eps}(\x,\y)=& \frac{|\y-\x-\eps \n(\y)|^3-|\y-\x+\eps \n(\y)|^3}{|\y-\x+\eps \n(\y)|^3|\y-\x-\eps \n(\y)|^3}\\
        =&\big( |\y-\x-\eps \n(\y)|-|\y-\x+\eps \n(\y)| \big) \times\\
        &\frac{\big( |\y-\x+\eps \n(\y)|^2+|\y-\x+\eps \n(\y)||\y-\x-\eps \n(\y)|+|\y-\x-\eps \n(\y)|^2 \big) }{|\y-\x+\eps \n(\y)|^3|\y-\x-\eps \n(\y)|^3}
    \end{align*}
    Using the fact that square root is $1/2$-H\"older ($a\geq b\geq 0, \sqrt{a}-\sqrt{b}\leq \sqrt{a-b}$)  and Lemma \ref{lem:reg1}, there exists $C>0$ such that
    $$\big| |\y-\x+\eps \n(\y)|-|\y-\x-\eps \n(\y)| \big|\leq 2\sqrt{\eps |\langle \x-\y,\n(\y)\rangle|}\leq C \sqrt{\eps} |\x-\y|$$
    Now we use the minoration of the denominator from Lemma \ref{lem:min}.
    Up to a global multiplicative constant $M$, we get,
    for any $0\leq \nu \leq 4$,
    \begin{align*}
        |f_{\eps}(\x,\y)|&\leq 4 C\frac{\sqrt{\eps} |\x-\y|}{|\x-\y|^{4-\nu}\eps^\nu}  \\
        &\leq M\frac{1}{\eps^\alpha |\x-\y|^{5/2-\alpha}}     
    \end{align*}
for any $-0.5 \leq \al \leq 3.5$
\end{proof}
Thanks to Lemma \ref{lem:reg2} with any $\alpha \in (1/2,1)$, there exists $C>0$ such that
$$|\int_{S} \big( \frac{\eps \langle \n(\y), \n(\x) \rangle}{|\y-\x+\eps \n(\y)|^3} - \frac{\eps \langle \n(\y), \n(\x) \rangle}{|\y-\x-\eps \n(\y)|^3} \big) \jtwo(\x) d\x|_{\R^3} \leq C \int_S \frac{\eps}{\eps^\al |x-y|^{5/2-\al}}|\jtwo|.$$
Using an Hardy-Littlewood-Sobolev inequality (e.g. \cite{han_hardylittlewoodsobolev_2016}) for $ 1\leq p <\infty$, there exists $C_\al >0$ such that
$$|\int_{S} \big( \frac{\eps \langle \n(\y), \n(\x) \rangle}{|\y-\x+\eps \n(\y)|^3} - \frac{\eps \langle \n(\y), \n(\x) \rangle}{|\y-\x-\eps \n(\y)|^3} \big) \jtwo(\x) d\x|_{\mathfrak{X}^p(S)}\leq C_\al \eps^{1-\al} |\jtwo|_{\mathfrak{X}^{1,2}(S)}.$$
Thus 
\begin{equation}
  \label{eq:NonnaPapera}
\int_{S}\langle \jone(\y)  \cdot \n(\x)\rangle \big( \frac{\eps \langle \n(\y), \n(\x) \rangle}{|\y-\x+\eps \n(\y)|^3}- \frac{\eps \langle \n(\y), \n(\x) \rangle}{|\y-\x-\eps \n(\y)|^3} \big) \jtwo(\x) d\x \xrightarrow{\mathfrak{X}^p(S)}0\end{equation}

In summary, Eq.~\ref{eq:surfint2} is the sum of two integrals converging
respectively as in Eq.~\ref{eq:Pluto} and \ref{eq:NonnaPapera}.  
Ultimately the ``first normal term'' converges to Eq.~\ref{eq:L2}. 

\subsubsection{Proof: Second normal term} \label{subsec:norm2}

The same reasoning just applied to integral \ref{eq:surfint2} also
applies to integral \ref{eq:surfint4}, 
$$\int_S \langle \jone(\y) \cdot \jtwo(\x) \rangle \frac{\langle \y-\x, \n(\x) \rangle \pm \eps \langle \n(\y), \n(\x) \rangle}{|\y-\x \pm \eps \n(\y)|^3} d\x,$$
which converges to
$$\int_S \langle \jone(\y) \cdot \jtwo(\x) \rangle \frac{\langle \y-\x, \n(\x) \rangle}{|\y-\x|^3} d\x,$$
i.e.~to Eq.~\ref{eq:L4}.

This concludes the proof of Theorem \ref{th:main}: one by one, in
Secs.~\ref{subsec:tang1}-\ref{subsec:norm2}, we have obtained all terms in
Eqs.~\ref{eq:L1}-\ref{eq:L4}.

Note that we do not expect $\int_S \langle \jone(\y) \cdot \jtwo(\x) \rangle \frac{\eps \langle \n(\y), \n(\x) \rangle}{|\y-\x+\eps \n(\y)|^3} d\x$ to go to 0 as that term is responsible for the magnetic field discontinuity.
But we are still able to use Lemma \ref{lem:reg2} to control the term $\frac{\eps}{|\y-\x+\eps \n(\y)|^3}-\frac{\eps}{|\y-\x-\eps \n(\y)|^3}$.

\begin{remark}
    \label{rmk:notLinfty}
    We do not expect $\L(\jone,\jtwo)$ to be in $L^\infty(S,\R^3)$. Indeed, $H^1(S)$ is not embedded in $L^\infty(S)$ for manifolds of dimension 2. For example, there is no constant $C>0$ such that $\big| \int_S \frac{1}{|\y-\x|}  \dive_\x(\pi_\x \jone(\y)) \jtwo(\x) d\x \big|_{L^\infty(S,\R^3)} \leq C |\jone|_{\mathfrak{X}^{1,2}(S)} $.
\end{remark}

\subsection{Justification from a 3D current modelisation}
\label{subsec:3D_current}
Recapitulating, the Laplace fore has been initially defined as 
the $\eps \to 0$ limit of the semi-sum of the magnetic field evaluated
at a distance $\eps$ away from the CWS, $S$, respectively
inward and outward (Eq.~\ref{eq:L_eps}). This was shown to either
be numerically costly or subject to numerical errors (Figure \ref{fig:B_norm}).

An expression (Eqs.~\ref{eq:L1}-\ref{eq:L4}) has then been derived in
Theorem \ref{th:main} for the Laplace force exerted by one current-sheet
on another, per unit length. The special case $\jone = \jtwo$ describes
the self-interaction of a current-sheet.

Both treatments relied on an intrinsically 2D model for the currents
on the CWS. A third approach is to treat the CWS as a
3D layer of infinitesimal thickness $\eps$.  
For some $\y \in S$ and if $\j$ is smooth enough, we could compute the
$\eps \to 0$ limit of
\begin{eqnarray*}
    \tilde \Leps(\j_\eps)(\y)=\int_{-\eps/2}^{\eps/2} \big[\j_\eps(\y+\eps_1\n(\y)) \times \B(\y+\eps_1 \n(\y)) \big] d\eps_1.
\end{eqnarray*}
Note that $\B(\y+\eps_1 \n(\y))$ is well-defined as we integrate on a 3D domain, and is given by:
\begin{eqnarray*}
  \B(\y+\eps_1 \n(\y)) = \int_S \int_{-\eps/2}^{\eps/2} \big[
    \j_\eps(\x+\eps_2 \n(\x)) \times \frac{\y-\x+ \eps_1 \n(\y)- \eps_2 \n(\x)}{|\y-\x+\eps_1 \n(\y)- \eps_2 \n(\x)|^3}
    \big]  dS(\x) d\eps_2.
\end{eqnarray*}
In order to approximate the 3D volume with a 2D current-sheet,
we suppose that $\forall z\in S$ and $\forall \eps'$, it is
$\j_\eps(\z+\eps' \n(\z))= \frac{\j(\z)}{\eps}$. Thus,
\begin{eqnarray*}
  \tilde \Leps(\j_\eps)=&\frac{1}{\eps^2} \int_{-\eps/2}^{\eps/2} \int_{-\eps/2}^{\eps/2} \big\{
  \int_S \j(\y)\times
  \big[
  \j(\x) \times \frac{\y-\x+\eps_1 \n(\y)- \eps_2 \n(\x)}{|\y-\x+\eps_1 \n(\y)- \eps_2 \n(\x)|^3}
  \big] dS(\x)
    \big\} d\eps_2 d\eps_1
\end{eqnarray*}
The quantity inside the brackets is very close to the one we got in
Theorem~\ref{th:main}, starting with Eqs.~\ref{eq:BS} and \ref{eq:L_eps}, 
except that we also have a contribution from $\eps_2 \n(\x)$.
It is possible to prove, using an argument similar to Lemma \ref{lem:reg2},
that replacing $\n(\x)$ with $\n(\y)$ does not change the limit.
The intuition is that for $\x$ close to $\y$, $\n(\x)$ is close to $\n(\y)$.
As a result, $\tilde \Leps(\j)$ has the same limit as $\Leps(\j)$ and the expression we found for the Laplace force (Eqs.~\ref{eq:L1}-\ref{eq:L4})
is consistent.

\section{Examples of cost functions}\label{sec:costs}

After having rigorously defined the Laplace force-density $\L(\j)(\y)$ that a
  current-sheet of density $\j$ exerts on itself at location $y$
  (Eqs.~\ref{eq:L1}-\ref{eq:L4} for $\jone = \jtwo = \j$),
  we now introduce some cost-functions to penalize high values of the force.

  Two main options are possible, and considered here:
  (1) penalizing high cumulative (or, equivalently, surface-averaged) forces,
  or (2)
  penalizing or even forbidding excessively high local maxima of the force. 
  Further variants are possible for specific
  force-components (e.g.~tangential or normal to the CWS) or a 
  weighted combination of them,
  with higher weights assigned to the engineeringly
  more demanding component, depending on the specific stellarator design.
  Such variants go beyond the scope of the present paper, and are left for
  future work.
 
A natural choice from the functional analysis point of view is to use
a penalization of the form
\begin{equation}
  |\L(\j)|_{L^p(S,\R^3)}=\big( \int_S |\L(\j)(\x))|^p d\x \big)^{1/p}.
  \label{eq:CumulCost}
  \end{equation}
The case $p=2$ is well-known: it represents the cumulative 
(or, barring a factor, the surface-averaged) root-mean-square force. 
Higher values of $p$ penalize more severely high values of
the Laplace force (i.e., large oscillations around the average norm). 
By contrast, low values of $p$
penalize the average norm of the Laplace force.

In principle it is also possible to use a $L^{\infty}$ cost, $\sup_S |\L(\j)|$,
but the domain might be smaller than $\mathfrak{X}^{1,2}(S)$.
However, such cost is not differentiable whenever
the maximum is reached at multiple locations.

The second option is to introduce the cost
\begin{equation}
    \label{eq:C_e}
    C_e(j)= \int_{S}f_e\left( \L(\j)(\x) \right) d\x
\end{equation}
as the surface integral of the local cost
\begin{equation}
  f_e(w)=\frac{\max(w-c_0,0)^2}{1-\frac{\max(w-c_0,0)}{c_1-c_0}}.
  \label{eq:f_e}
\end{equation}
\begin{figure}
    \includegraphics[scale=0.5]{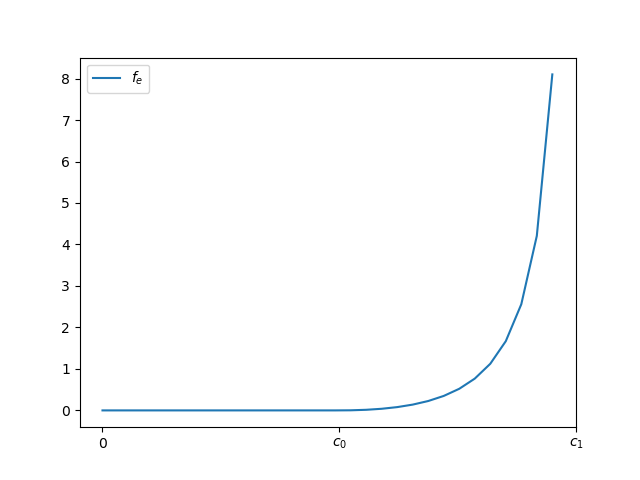}
    \captionof{figure}{Plot of the local cost $f_e$ as a function of
      the local force $w$. Note that $f_e$ diverges at $c_1$ and vanishes
      in $[0,c_0]$.}
    \label{fig:f_e}
\end{figure}
The domain for this cost is not the entire space $\mathfrak{X}^{1,2}(S)$,
but this cost captures more effectively the engineering
constraints of building a high-field stellarator: the mechanical properties of
support-structures and materials are such that forces below a threshold $c_0$
are negligible, forces higher and higher
than $c_0$ should be penalized more and more, and forces
above a second, ``rupture'' threshold $c_1$ should be completely forbidden. 
Indeed, the local cost $f_e$ evolves with the local force $w$
as desired, as illustrated in Figure \ref{fig:f_e}.

\begin{remark}
  It is unclear whether a minimizer exists
  in $\mathfrak{X}^{1,2}(S)$ for the costs discussed.
  As a consequence, a good practice is to add a regularizing
  term $|\j|_{\mathfrak{X}^{1,2}(S)}$.
\end{remark}

\section{Numerical simulations}
\label{sec:numerics}
\subsection{Setup}

To test our force-reduction method, we ran simulations for the NCSX
stellarator equilibrium known as LI383 \cite{Zarn}.

As mentioned in the Introduction, the costs defined in
Sec.~\ref{sec:costs} are easily added to the cost-function in any stellarator
coil optimization code. In our case such code was a new incarnation of
{\tt REGCOIL}, which we rewrote in python instead of fortran, and compiled with
the Just In Time compiler Numba \cite{Numba}. 
  For the most part the new code is conceptually identical to {\tt REGCOIL},
  except that it uses Eq.~A5 of Ref.~\cite{Landreman} in lieu of its normal,
  single-valued component (Eq.~A8 from the same paper). Eq.~A5 would be
  numerically unstable if derivatives were taken by finite differences,
but can be used here because we compute the derivatives explicitly.
We benchmarked the new code and found it to 
agree with the original {\tt REGCOIL} to within 7 significant digits.

The surface-current $\j$ is divergence-free and thus taken in the form
\begin{equation}
    \label{eq:expression_j}
    G  \frac{\partial {\bf r'}}{\partial \theta} -I \frac{\partial {\bf r'}}{\partial \zeta}+ \frac{\partial \Phi'}{\partial \zeta'} \frac{\partial {\bf r'}}{\partial \theta'}-\frac{\partial {\bf r'}}{\partial \zeta'}\frac{\partial \Phi'}{\partial \theta'}.
\end{equation}
Here $\theta$ and $\zeta$ are the poloidal and toroidal angle,  
$G$ and $I$ are optimization inputs (net poloidal and toroidal currents)
and the current potential $\Phi$ is decomposed in a 2D Fourier basis.
\begin{center}
    \includegraphics[scale=0.5]{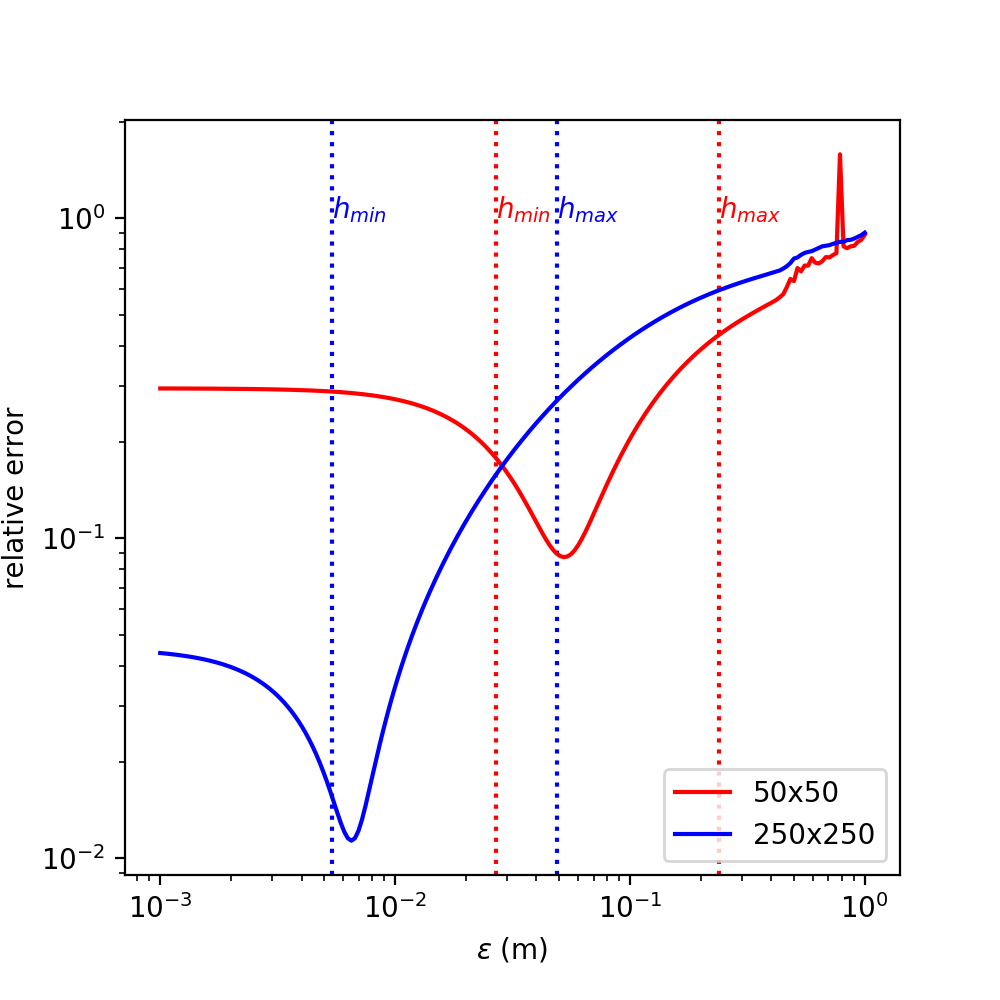}
    \captionof{figure}{Convergence of $\Leps$ toward $\L$ for NCSX, for two different grids. Convergence stops when $\eps \lesssim h$, due to a numerical
    error in $\B$ (Figure \ref{fig:B_norm}).}
    \label{fig:L_eps}
\end{center}
\begin{center}
    \includegraphics[width=\textwidth]{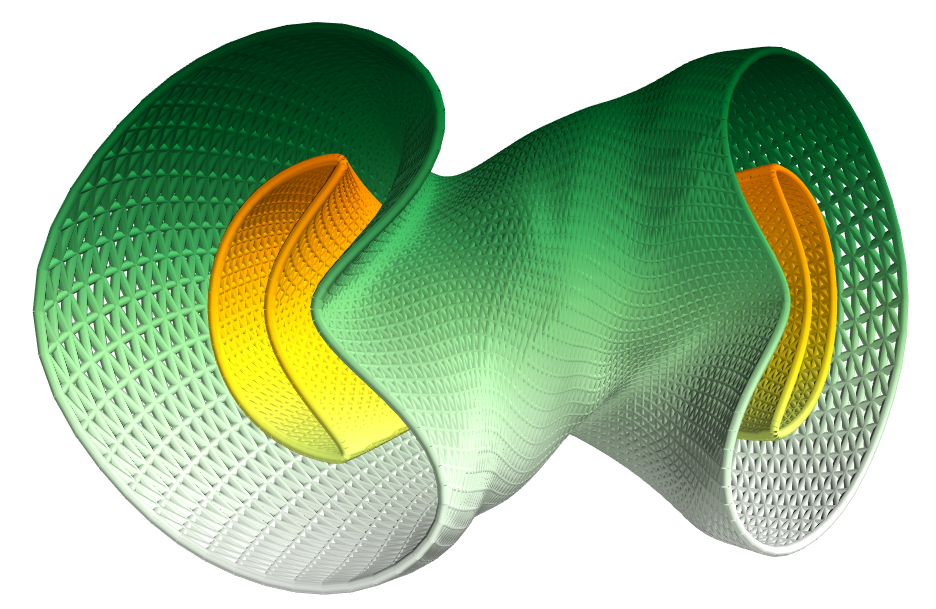}
    \captionof{figure}{NCSX LI383 plasma surface in orange and CWS in green.}
    \label{fig:cws}
\end{center}
Figure \ref{fig:L_eps}
illustrates how $\Leps$ converges to $\L$ (or, equivalently, the relative
error vanishes) as $\eps \to 0$. 
We recall that the numerical evaluation of $\Leps$ involves three characteristic distances $h$ (discretisation length of the mesh), $\eps$, and $d_B$ (characteristic distance of variation of the magnetic field).
For reference we computed $\L$ on the same mesh (that is, for the same $h$), 
and obviously $d_B$ was also the same.

We observe that the error decreases with $\eps$, as expected,
but when $\eps \lesssim h$ the convergence stops and the error grows again.
This is a consequence of the calculation of $\B$ not being
accurate anymore, for $\eps \lesssim h$ (see Figure \ref{fig:B_norm}).
Note that the reference value itself is an approximation. As we do not have an
analytic expression, we also used a discretisation. The relative error is
computed in $L^2$ norm.

In all simulations presented here, both the CWS and the plasma surface
were discretized as $64\times64$ meshes in the poloidal$\times$toroidal
direction. The two surfaces are rendered in 3D in Figure \ref{fig:cws}.

The single-valued current potential $\Phi$ from which $\j$ descends
is represented by 8 or 12 harmonics in each direction. As we do not impose
stellarator symmetry, we use as a basis the functions 
$$\sin (k\theta +l \zeta), \quad \cos(k\theta +l \zeta)$$
with $0 \leq k \leq N$ and $-N\leq l \leq N$.
Since for $k=0$ we can restrict to $0 < l$, the total number of
Degrees Of Freedom (DOF) is $2 [(2N+1)N +N]$.

Thus $N=$8 harmonics in each direction correspond to 288 DOF, and $N=$12
yields 624 DOF. 
Better results can be achieved with more harmonics, as shown in
Figure \ref{fig:regcoil_compare1}. However, a finer mesh is required,
making the problem computationally more expensive.

The optimization is performed by conjugate gradient. 
With our implementation, a single evaluation of the gradient lasts
approximately 2 minutes on a small cluster of 64 cores. 
The full optimization can last a few days.

\subsection{Adding force minimization and improving regularization in 
  {\tt REGCOIL}}   \label{subsec:AddForceMin}

We propose to integrate the costs introduced above in the same optimization
scheme as {\tt NESCOIL} \cite{Merkel} and {\tt REGCOIL} \cite{Landreman}.

As a reminder, {\tt NESCOIL} seeks the current of density $\j$, on a fixed $S$,
that maximizes magnetic field accuracy on the plasma boundary $S_P$ (hence,
indirectly, in the plasma). It does so by minimizing the ``plasma-shape
objective'' or ``field accuracy objective''
\begin{equation}
  \chi^2_B=\int_{S_P} \langle \B(\x) \cdot \n(\x) \rangle^2 dS(\x).
  \label{eq:chiB}
\end{equation}

{\tt REGOIL}, instead, compromises between field accuracy and coil
simplicity by minimizing $\chi^2_B+\lambda \chi^2_\j$, where $\lambda$ is a
weight and the ``current-density objective'' or
``regularizing term'' $\chi^2_\j$ is a penalty on high values of
$\j$, in the sense of the $L^2$ norm: 
\begin{equation}
  \chi^2_\j= \int_{S} |\j|^2 dS.
  \label{eq:chij}
\end{equation}
Heavier weighting makes $\Phi$ (hence $\j$, hence the coils) more regular,
but at the expense of reduced field accuracy. Such cost is identical to
$\chi_K^2$ of Ref.~\cite{Landreman}, but is renamed $\chi_\j^2$
for consistency of notation with another regularizing term that
we need to introduce: 
\begin{equation}
  \chi_{\nabla\j}^2 =
  \int_{S} (|\nabla \j_x|^2+|\nabla \j_y|^2+|\nabla \j_z|^2) dS.
  \label{eq:chinabla}
\end{equation}
This new term is motivated by Theorem~\ref{th:main}: 
as the Laplace force can only be defined for $\j \in \mathfrak{X}^{1,2}(S)$,
it is natural to add a penalization on the gradient of $\j$ and not just
on $\j$. Basically we are replacing the $L^2$ norm of $\j$ 
with the $H^1$ norm of $\j$.

Here we propose to further generalize the {\tt REGCOIL} cost function to
\begin{equation}
  \chi ^2 = \chi_B ^2 + \lambda_1 \chi_\j ^2 + +\lambda_2 \chi_{\nabla\j}^2 +
  \gamma \chi_F ^2,
  \label{eq:chi}
\end{equation}
where $\chi_F^2$ is a ``force objective'' that
penalizes strong forces on the current-sheet, i.e.~among 
the coils. Per the discussion in Sec.~\ref{sec:costs}, possible definitions
include:
\begin{align}
  \chi_F ^2 &= |\L(\j)|_{L^2(S,\R^3)}^2 =\int_{S} |\L(\j)|^2 dS
  \label{eq:chiF1}
  \\ 
  \chi_F ^2 &= C_e = \int_{S} f_e(|\L(\j)|) dS
  \label{eq:chiF2}
\end{align}

with $f_e$ defined as in Eq.~\ref{eq:f_e} and plotted in
Fig.~\ref{fig:f_e}. As stress limits, here we set $c_0=5\cdot 10^6$ Pa 
and $c_1=10^7$ Pa.

\subsection{Numerical results}

There is obviously a trade-off between conflicting objectives
  in Eq.~\ref{eq:chi}. 
To study that, we compared a {\tt REGCOIL}-like case
with a force-minimization case.
  
  In the {\tt REGCOIL}-like case (curves in Fig.~\ref{fig:regcoil_compare1})
  we fixed $\lambda_2 = \gamma = 0$ and minimized
  $\chi^2 = \chi_B ^2 + \lambda_1 \chi_\j ^2$
  for various choices of $\lambda_1$. By this scan we re-obtained 
  the well-known trade-off between $\chi_B^2$ and $\chi_\j^2$
  (or, equivalently, field-accuracy and coil-simplicity) exhibited by 
  {\tt REGCOIL} \cite{Landreman} (not plotted).
  Interestingly, we also found a trade-off between $\chi_B^2$ and $\chi_F^2$,
  even though $\chi_F^2$ was not part of the  
  $\chi_B ^2 + \lambda_1 \chi_\j ^2$ minimization.
  The trade-off between these global quantities 
  is plotted in Fig.~\ref{fig:regcoil_compare1}a, and a
  trade-off between related, local quantities
  is plotted in Fig.~\ref{fig:regcoil_compare1}b. In other words,
  more (less) accurate solutions tend to be subject to higher (lower) forces,
  even when forces are not accounted in the minimization. 

  In the force-minimization case, instead,
  we fixed $\lambda_1 = \lambda_2 = 0$ and minimized
  $\chi^2 = \chi_B ^2 + \gamma \chi_F ^2$ for various choices of $\gamma$.
  Not surprisingly, we found a trade-off between $\chi_B^2$ and $\chi_F^2$ 
  (symbols in Fig.~\ref{fig:regcoil_compare1}).
  Interestingly, we also found a trade-off between $\chi_B^2$ and $\chi_\j^2$,
  even though $\chi_\j^2$ was not part of the minimized cost. 
  This suggests that $\chi_F$ has a regularizing effect on $\j$, as it will
  become apparent in Fig.~\ref{fig:multi_plots} and \ref{fig:multi_plots2}.
  
Finally, Fig.~\ref{fig:regcoil_compare1} confirms
that a higher number of Fourier harmonics and hence of DOF reproduces the
magnetic field more accurately. This is why for the remainder of the article we
adopt the higher number of DOF, 624.

Also, we no longer scan the weights, but fix them to yield reasonable
compromises between field accuracy, current regularization
and/or force minimization.  In particular,
calculations were performed for the following four
choices of weights and $\chi²_F$ in Eq.~\ref{eq:chi}:

\begin{align}
  \label{tab}
\begin{tabular}{c|cccc} 
  Case &  $\lambda_1$         & $\lambda_2$ &   $\gamma$      &  $\chi²_F$  \\
       & (T$^2$ m$^2/$A$^2$)  & (T$^2$ m$^4/$A$^2$) &  (T$^2$/Pa$^2$) &      \\
  \hline
  1    &  $1.5\cdot 10^{-16}$ &0  &   0          &     0        \\
  2    &  0                   &0 &   $10^{-17}$  & $|\L(\j)|_{L^2(S,\R^3)}^2$ \\
  3    &  0                   &0 &   $10^{-16}$  &   $C_e$     \\
  4    &  $10^{-19}$           & 10$^{-19}$ &   $10^{-16}$  &   $C_e$       
\end{tabular}
\end{align}
Case 1 is basically {\tt REGCOIL}, whereas
case 2 and 3 are effectively {\tt NESCOIL} but with minimized forces, according
to two different force metrics. Finally, case 4 explicitly
combines force minimization with regularization, but
in a broader sense compared to {\tt REGCOIL}, as discussed in connection
with Eq.~\ref{eq:chinabla}. 

The results for these four cases are plotted in
Fig.~\ref{fig:regcoil_compare2} (circles)
and compared with {\tt REGCOIL} results (curve).
In particular Fig.~\ref{fig:regcoil_compare2}a refers to surface-integrated,
``global'' objectives, and Fig.~\ref{fig:regcoil_compare2} to ``local'' maxima. 
Note the logarithmic plots. As expected, case 1 agrees with {\tt REGCOIL}. 
Case 2 (defined in terms of the ``global'' $|\L(\j)|_{\mathfrak{E}^2}^2$)
overperforms in the ``global'' Fig.~\ref{fig:regcoil_compare2}a, as expected.
Actually, it performs better than {\tt REGCOIL} even in terms of local metrics
(Fig.~\ref{fig:regcoil_compare2}b). 
Compared to {\tt REGCOIL}, peak-forces are reduced in cases 3 and 4
(Fig.~\ref{fig:regcoil_compare2}b), and remain lower than the chosen $c_1$, 
as is expected from the definition of $C_e$ and $f_e$
(Eqs.~\ref{eq:C_e}-\ref{eq:f_e})
However, this happens at the expense of higher
cumulative forces (Fig.~\ref{fig:regcoil_compare2}a). 

Details on the four cases are presented in Fig.~\ref{fig:multi_plots}
and  \ref{fig:multi_plots2}. Columns from left to right refer to cases
from 1 to 4. From top to bottom, the rows in Fig.~\ref{fig:multi_plots}
present contours of the norm of $\j$, the component of the magnetic field
normal to $S$ and the norm of the Laplace force, as functions of the poloidal
and toroidal angles. The two rows in Fig.~\ref{fig:multi_plots2}
present the force components normal and tangential to $S$.

As anticipated, case 2 is as regular as case 1, in spite of its $\chi^2$
not containing a regularizing objective.
By contrast, case 3 reproduces the field with high accuracy and exhibits
reduced peak forces, as expected from the definition of $C_e$, 
but with a complicated current-pattern. That is ameliorated by adding
some regularization: case 4 is the best compromise between coil simplicity
  (first row in  Fig.~\ref{fig:multi_plots2}), field accuracy (second row) and
  reduced forces (third row).

Incidentally all cases, including {\tt REGCOIL} (case 1) and
  the magnetically most accurate case 3, 
  exhibit residual field errors of up to 60 mT. Lower errors can be achieved
  by adopting a higher number of DOF, as is intuitive and suggested by
  Fig.~\ref{fig:regcoil_compare1}, but this is computationally more intensive and
beyond the scope of the present paper.

From the point of view of the surface-integrated or surface-averaged forces,
the best result in Fig.~\ref{fig:multi_plots} is a modest 
reduction by 5\% for case 2, relative to {\tt REGCOIL}. 
From the point of view of peak forces, however, the best result in 
Fig.~\ref{fig:multi_plots} is a 
reduction by 40\% for case 4, relative to {\tt REGCOIL}. 
Correspondingly, the peak tangential force is reduced by 50\% and the
peak normal force by 20\% (Fig.~\ref{fig:multi_plots2}). Note that maxima for
different components occur at different toroidal and poloidal locations.

More dramatic reductions were obtained in Fig.~\ref{fig:regcoil_compare1}, especially in peak forces. 
However, they were obtained for low-accuracy cases on the top left of
Fig.~\ref{fig:regcoil_compare1}b: a stellarator with those characteristics
would suffer from very low coil-forces, but it would also be a poor match
of the target field.

\section{Summary, conclusions and future work}
        To summarize, force-minimization is an important aspect of
          stellarator coil-optimization, especially for future high-field stellarators.
          In the present article we 
rigorously proved in Sec.~\ref{subsec:Theo} that
the Laplace force exerted by a surface-current onto one another can be written as in
Eqs.~\ref{eq:L1}-\ref{eq:L4}). From that, one can calculate the auto-interaction $\L(\j)$ 
of a current-distribution with itself, and distill that information in a single scalar.
Possible metrics were discussed in Sec.~\ref{sec:costs}, and two of them were used for
detailed numerical calculations: two possible ``force objectives'' (Eqs.~\ref{eq:chiF1} and
\ref{eq:chiF2}) were added to the cost function of the well-known {\tt REGCOIL} code
\cite{Landreman}. In addition, the $L^2$ norm of $\j$ was replaced  
with the $H^1$ norm of $\j$, for reasons explained in Secs.~\ref{subsec:notations} and
\ref{subsec:AddForceMin}.

This approach permitted to
simultaneously optimize the coils of the NCSX stellarator 
for high magnetic fidelity, high regularity and low forces, e.g.~40\% lower
peak forces compared to {\tt REGCOIL} for similar plasma shape accuracy,
slightly higher average current and lower current peak
(Figs~\ref{fig:regcoil_compare1}-\ref{fig:multi_plots2}).

Force reduction is an important criterion in stellarator
optimization, and future high-field designs might benefit from our approach. 

In the present work the Coil Winding Surface (CWS) was fixed.
Future shape optimization of the CWS, inspired by Ref.~\cite{paul2020adjoint}, 
is expected to further reduce the coil-forces.
Moreover, the constraints $c_0$ and $c_1$ on penalized and
forbidden forces (Fig.~\ref{fig:f_e}) were fixed. Future work could  
impose stricter constraints and tailor them differently for normal and
longitudinal forces, as they tend to differ (Fig.~\ref{fig:multi_plots2}) 
and obey to different engineering and material constraints. 
Finally, the new code is magnetically as accurate as {\tt REGCOIL}, when
operated with the same number of Fourier harmonics and hence of degrees of
freedom (DOF). However, it is also
significantly slower as a result of force minimization. 
Optimizing the code for speed would allow to retain a higher number of DOF
and achieve higher field accuracy for the same amount of cpu time, while at
the same time optimizing for simplicity and force reduction.

\section{Acknowledgements}

The authors thank Mario Sigalotti for the fruitful discussions and for 
carefully reading the manuscript. 
This work has been partly supported by Inria's Action Exploratoire
  StellaCage.

\begin{figure}
    \includegraphics[width=\textwidth]{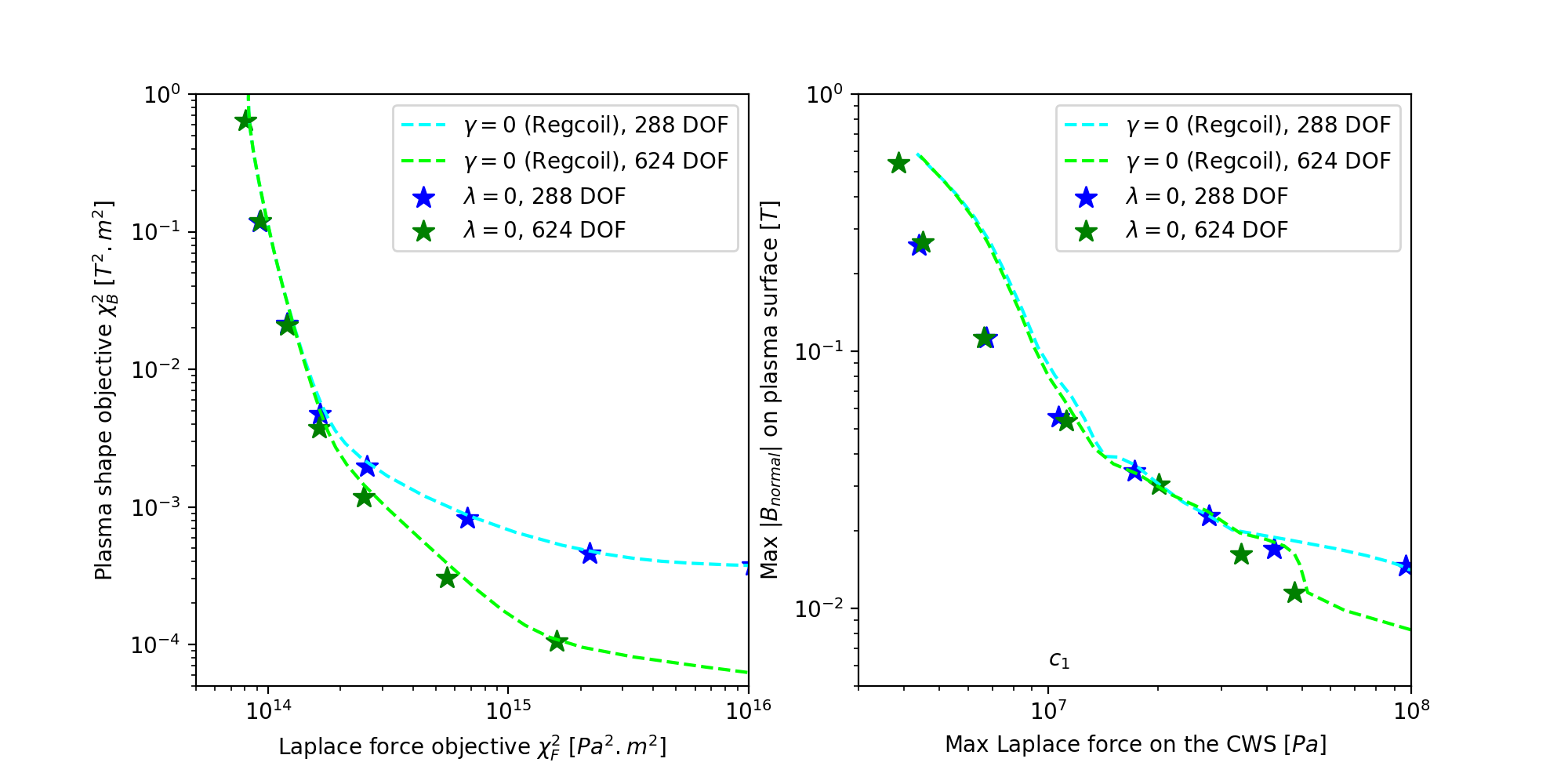}
    \captionof{figure}{
      (a) Trade-off between plasma shape accuracy and Laplace force metrics
      (as defined in Eq.~\ref{eq:chiF1}) 
      for different weightings in Eq.~\ref{eq:chi} and different
      numbers of harmonics, and thus of Degrees of Freedom (DOF). Such trade-off, 
      expected when optimizing a linear combination of $\chi^2_B$ and $\chi^2_F$ (symbols),
      is also observed in the minimization of $\chi^2_B$ and $\chi^2_j$ (curves).
      (b) Similar trade-off between maximum field and 
      maximum Laplace force (Eq.~\ref{eq:chiF2}).}
    \label{fig:regcoil_compare1}
\end{figure}

\begin{figure}
    \includegraphics[width=\textwidth]{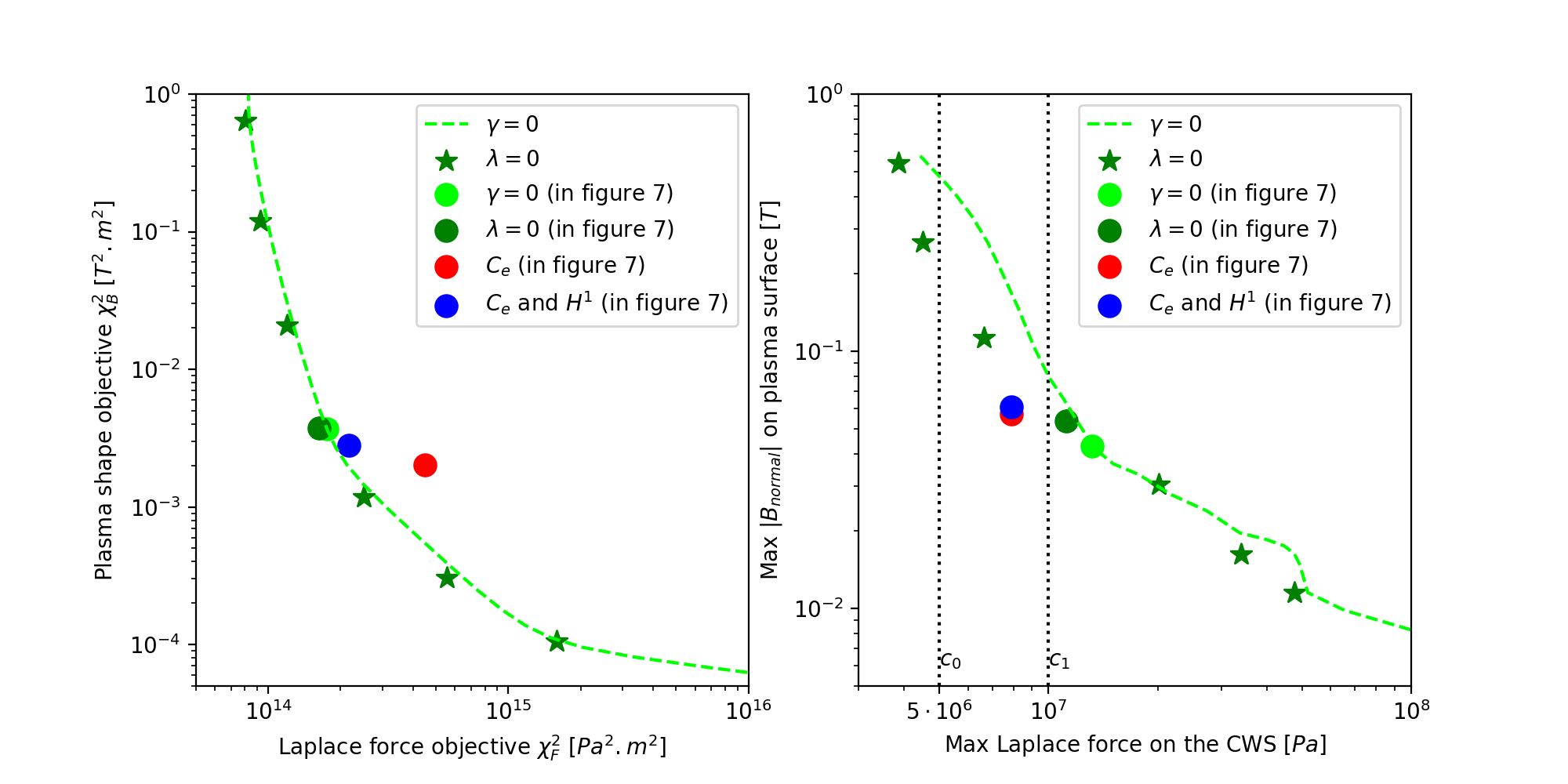}
    \captionof{figure}{
      Trade-offs between: (a) plasma shape accuracy and Laplace force metrics
      (Eq.~\ref{eq:chiF1}) and (b) maximum field and 
    maximum Laplace force (Eq.~\ref{eq:chiF2}).
    Unlike Fig.~\ref{fig:regcoil_compare1},
    all simulations here used 624 DOF. Circle symbols correspond to the four cases
    discussed in Sec.~\ref{sec:numerics} and presented 
    in Figure \ref{fig:multi_plots}. As expected, the $C_e$ cases (red and blue)
    fall between the penalized and forbidden forces $c_0$ and $c_1$ defined in
    Fig.~\ref{fig:f_e}, marked here by vertical dotted lines.}
    
    \label{fig:regcoil_compare2}
\end{figure}

\begin{figure}
    \includegraphics[width=1\textwidth]{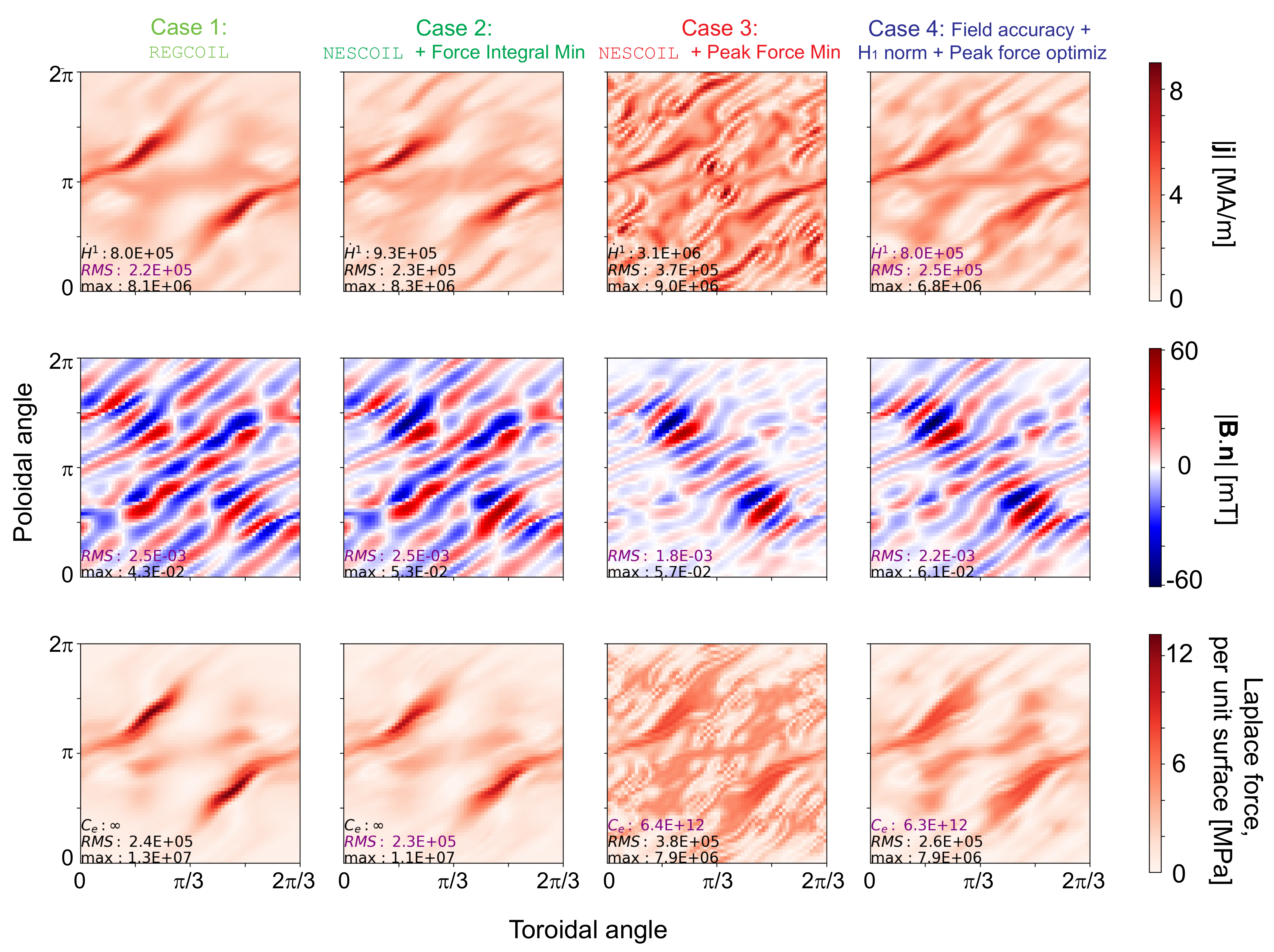}
    \captionof{figure}{ Results of minimizing Eq.~\ref{eq:chi} for NCSX, 
      for four different weight choices (Eq.~\ref{tab}).
      624 DOF are used for $\j$ in every simulation.
      From top to bottom
      the three rows refer respectively to the results for simultaneous
      current regularization (if any), field accuracy and
      force-minimization (if any).
      Shown in the legends are the 
      Root Mean Square (RMS) surface-averages and local maxima of the
      quantities plotted, as well as the $H^1$ norm of $\j$ and
      $C_e$ force metric (Eq.~\ref{eq:chiF2}). 
      The quantities actually minimized are marked in purple.}
    \label{fig:multi_plots}
\end{figure}

\begin{figure}
    \includegraphics[width=1\textwidth]{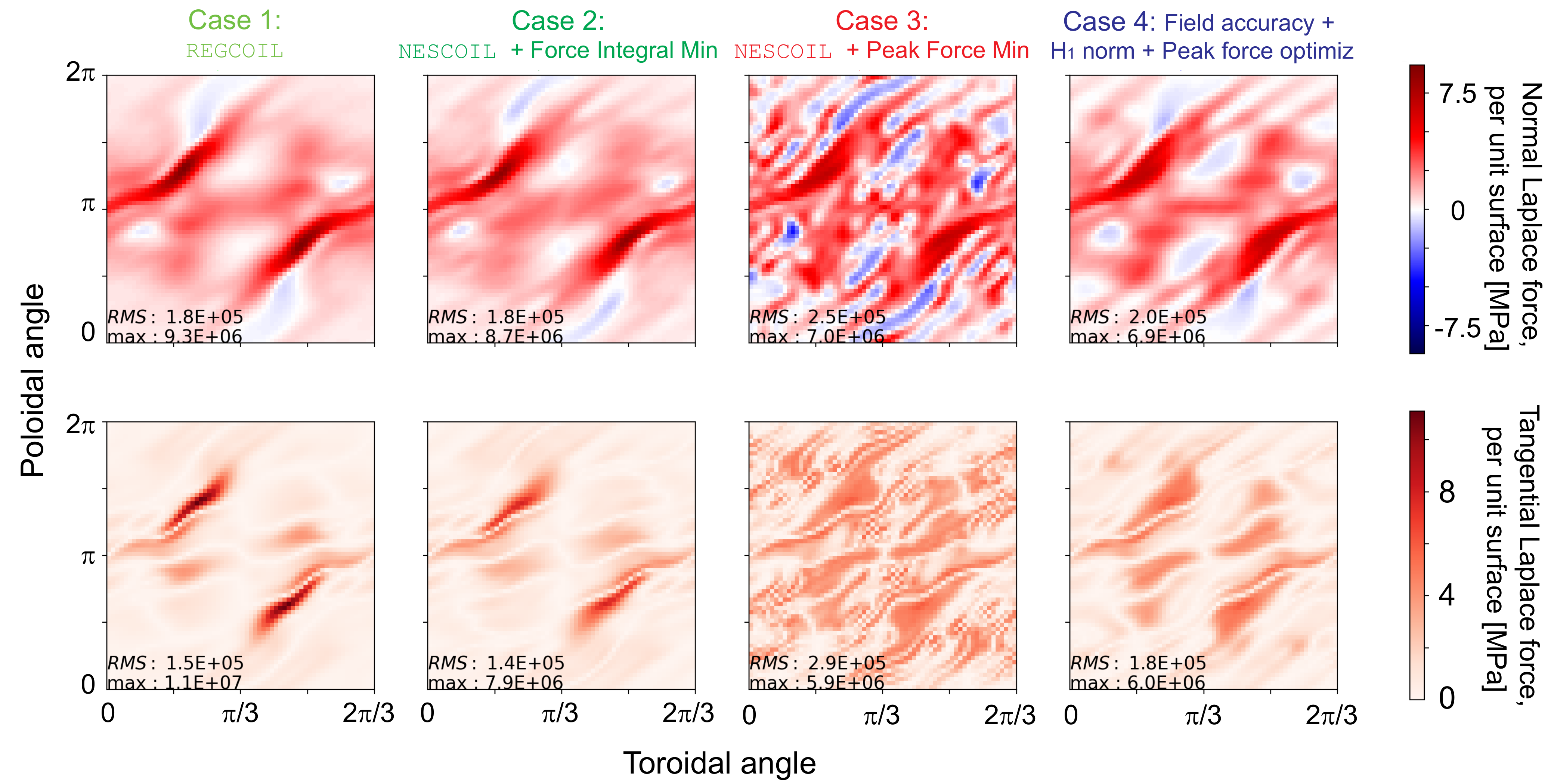}
    \captionof{figure}{Tangential and normal components of the Laplace forces of the simulations in Figure \ref{fig:multi_plots}.}
    \label{fig:multi_plots2}
\end{figure}

\begin{figure}
    \includegraphics[width=\textwidth]{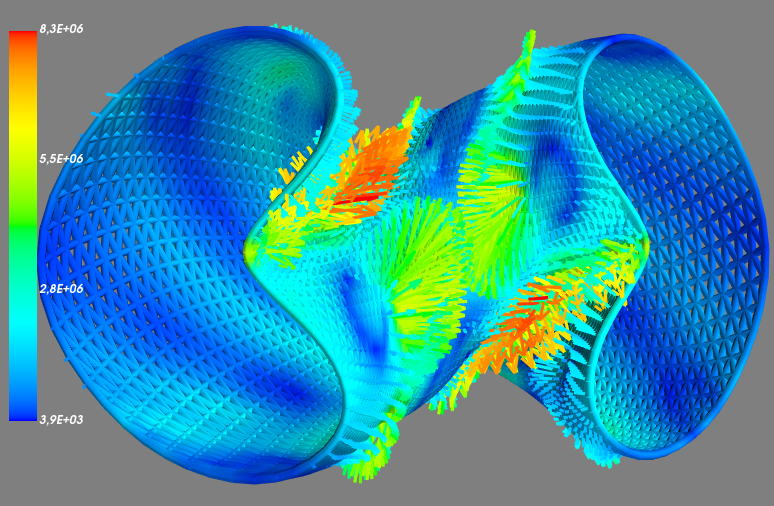}
    \captionof{figure}{The Laplace forces from the last column of Figure \ref{fig:multi_plots}. The unit for the pressure is Pascal.}
    \label{fig:3D_Visualization}
\end{figure}
\clearpage
\bibliographystyle{unsrt} 
\bibliography{fusion}
\end{document}